\documentclass[11pt]{amsart}

\usepackage{amssymb, amsthm, amsmath}
\usepackage{xcolor}
\newtheorem{theorem}{Theorem}[section]

\newtheorem{lemma}[theorem]{Lemma}
\newtheorem{corollary}[theorem]{Corollary}

\newtheorem*{theorem*}{Theorem}{\bf}{\it}
\newtheorem*{proposition*}{Proposition}{\bf}{\it}
\newtheorem*{observation*}{Observation}{\bf}{\it}
\newtheorem*{lemma*}{Lemma}{\bf}{\it}

\theoremstyle{definition}

\theoremstyle{remark}
\newtheorem{remark}[theorem]{Remark}

\renewcommand{\tilde}{\widetilde}

\def\XXint#1#2#3{{\setbox0=\hbox{$#1{#2#3}{\int}$ }
\vcenter{\hbox{$#2#3$ }}\kern-.6\wd0}}

\begin{document}
\title[]{Nodal sets of Laplace eigenfunctions: polynomial upper estimates of the Hausdorff measure.}

\author{Alexander Logunov}
\address{School of Mathematical Sciences, Tel Aviv University, Tel Aviv 69978, Israel}
\address{Chebyshev Laboratory, St. Petersburg State University, 14th Line V.O., 29B, Saint Petersburg 199178 Russia}
\email{log239@yandex.ru}

%\keywords{Harmonic functions, Nodal set}
%\subjclass[2010]{31B05}

\begin{abstract}
 Let $\mathbb{M}$ be a compact $C^\infty$-smooth Riemannian manifold of dimension $n$, $n\geq 3$, and let $\varphi_\lambda: \Delta_M \varphi_\lambda + \lambda \varphi_\lambda = 0$ denote the Laplace eigenfunction on $\mathbb{M}$ corresponding to the eigenvalue $\lambda$. We show that  $$H^{n-1}(\{ \varphi_\lambda=0\}) \leq C \lambda^{\alpha},$$
where $\alpha>1/2$ is a constant, which depends on $n$ only, and $C>0$ depends on $\mathbb{M}$ .
This  result is a consequence of our study of  zero sets of harmonic functions on $C^\infty$-smooth  Riemannian manifolds. We develop a technique of propagation of smallness for solutions of elliptic PDE that allows us to obtain local bounds from above for the volume of the nodal sets  in terms of the frequency and the doubling index.
% We obtain partial positive answers to the question: Is the frequency additive in some sense?
 \end{abstract}
\maketitle
%%%%%%%%%%%%%%%%%%%%%%%%%%%%%%%%%%%%%%%%%%%%%

\section{Preliminaries} \label{sec:di}

  Yau conjectured  that the Laplace eigenfunctions $\varphi_\lambda: \Delta \varphi_\lambda + \lambda \varphi_\lambda=0$ on a compact $C^\infty $-smooth Riemannain manifold $W$ of dimension $n$ (without  boundary) satisfy
$$c \lambda^{1/2} \leq H^{n-1}(\{ \varphi_\lambda=0\}) \leq C \lambda^{1/2}, $$
where $H^{n-1}(\cdot)$ denotes the $(n-1)$ dimensional Hausdorf measure, positive constants $c,C$ depend on the Riemannian metric and on the manifold only.
This conjecture was proved for  real-analytic manifolds by Donnelly and Fefferman (\cite{DF}). 
For non-analytic manifolds the best-known upper estimate in dimension $n=2$ was $H^{1}(\{ \varphi_\lambda=0\}) \leq C \lambda^{3/4}$ due to Donnelly and Fefferman (\cite{DF1}), different proof for the same bound was given by Dong (\cite{D}). Recently this bound was refined to  $C \lambda^{3/4-\varepsilon}$ in \cite{LM}.

 In higher dimensions the estimate $H^{n-1}(\{ \varphi_\lambda=0\}) \leq C\lambda^{C \sqrt \lambda }$  by Hardt and Simon (\cite{HS}) was the only known upper bound till now. We prove that

 $$H^{n-1}(\{ \varphi_\lambda=0\}) \leq C \lambda^{\alpha},$$
where $\alpha>1/2$ is a constant, which depends on $n$ only, $C$ depends on $M$.
  
 This estimate will follow from an estimate (Theorem \ref{thv}) for harmonic functions, which bounds the volume of the nodal set in terms of the frequency function (or the doubling index).

 There is a standard trick that allows to pass from Laplace eigenfunctions to harmonic functions:
one can add an extra variable  $t$ and consider a function $$u(x,t)=\varphi(x)\exp(\sqrt \lambda t),$$
which appears to be a harmonic function on the product manifold $W\times \mathbb{R}$.

  Let $M$ be a $C^\infty$-smooth  Riemannian manifold (non-compact and with no boundary), endowed with metric $g$. Consider a
 point $p \in M$ and a harmonic function $u$ (with respect to $g$) on $M$.  By $B_g(p,r)$ we will denote a  geodesic ball with center at point $p$ and radius $r$. Define $H(r) = \int\limits_{\partial B_g(p,r) } u^2 dS_r $, where $S_r$ is the surface measure  on $\partial B_g(p,r)$ with respect to $g$. 
 We will  always assume that  $r$ is smaller than the injectivity radius.

 \textbf{Definition.} The frequency function of a harmonic function $u$ is defined by
$$\beta(r):=\frac{r H'(r)}{2H(r)}. $$
   We remark that this definition is slightly different from the standard one, since we don't normalize $H(r)$ by the surface measure $|S_r|$. See \cite{HL} for a friendly introduction to frequency and also  \cite{GL}, \cite{L} for applications to nodal sets. In dimension two understanding of nodal sets of harmonic functions is better due to complex analysis techniques and topological reasons, see \cite{NPS}.

 We will work only on a bounded subset of $M$:
 fix a point $O$ on $M$ and  assume hereafter that $B_g(p,r) \subset B_g(O,1)$.

 The frequency is almost monotonic in the following sense (see Remark (3) to Theorem 2.2 in \cite{M}):
\begin{lemma}

 For any $\varepsilon>0$ there exists $R_0= R_0(\varepsilon,M,g,O)$ such that
\begin{equation} \label{b1}
\beta(r_1) \leq (1+\varepsilon) \beta(r_2)
\end{equation}
for any $r_1,r_2$: $0< r_1< r_2< R_0$.
\end{lemma}

 %There exists $c>0$ and $R>0$, depending  on  $M,g$ only, such that: $ e^{c r^2}\frac{r H'(r)}{H(r)}$ is  monotonic and increasing for $r < R$. Therefore

 One can estimate the growth of $H(r)$ in terms of the frequency in view of the integral formula:
$$ \frac{H(r_2)}{H(r_1)} = \exp(2 \int_{r_1}^{r_2} \beta(r) d \log r).$$
\begin{corollary} \label{logc}
 $(\frac{r_2}{r_1})^{2\beta(r_1)/(1+\varepsilon)}\leq \frac{H(r_2)}{H(r_1)} \leq (\frac{r_2}{r_1})^{2\beta(r_2)(1+\varepsilon)}$.
\end{corollary}
 Sometimes we will specify the center of the ball and our choice of the function $u$ and write $\beta(p,r) $ and $H(p,r)$ or $\beta_u(p,r) $ and $H_u(p,r)$ in place of $\beta(r)$ and $H(r)$.

 We  need a standard elliptic estimate that compares  $L^\infty$  and $L^2$ norms of harmonic functions on concentric geodesic spheres: for any $\varepsilon \in (0,1)$ there exists a constant $C_1 = C_1(\varepsilon, M, g,O)>0$ such that 
\begin{equation} \label{h1}
 \sup \limits_{\partial B_g(p,r)} |u|^2 \leq C_1 \frac{H(r(1+\varepsilon))}{r^{n-1}}
\end{equation} 
for $r\leq R_0$, where  $R_0=R_0(M,g,O)>0$.
 
  The reverse estimate  holds for arbitrary continuous functions on $M$:
\begin{equation}  \label{h2}
 H(r)\leq C_2(M,g,O) r^{n-1} \sup \limits_{\partial B_g(p,r)} |u|^2, 
\end{equation} 
 where $C_2(M,g,O)$ is  a positive constant such that the whole surface measure of a geodesic sphere $|S_r| \leq r^{n-1} C_2(M,g) $, $r\leq R_0$.

  Let us consider normal coordinates in a geodesic ball $B_g(O, R)$, where $R$ is a sufficiently small number. In these coordinates we will treat the Laplace operator as an elliptic operator in a fixed domain in $\mathbb{R}^n$, say, a  cube $Q$. We will identify $O$ with the origin and denote the ordinary Euclidean distance by $d(x,y)$ and the Riemannian distance by $d_g(x,y)$.
 Let  $\varepsilon>0$ be a small number. We will assume hereafter that 
 \begin{equation} \label{eqm}
\frac{d_g(x,y)}{d(x,y)} \in (1-\varepsilon, 1+ \varepsilon)
\end{equation}
 for  points $x,y$ in $B_{g}(O,R_0)$: $x \neq y$, where $R_0=R_0(\varepsilon, M,g,O)>0$. The existence of such $R_0$ for any $\varepsilon$ is provided by the choice of the normal coordinates.

 For the purposes of the paper it will be more convenient to work with a notion similar to the frequency: so-called doubling index, which deals with $L^\infty$ norms in place of $L^2$ and Euclidean balls in place of geodesic balls.
 For a given  ball  $B$ (ball in standard Euclidean metric) define the doubling index $N(B)$ by $2^{N(B)}= \frac{ \sup\limits_{2B} |u| }{\sup\limits_{B} |u|}$. Given a positive number $r$  we denote by $rB$ the homothety image of $B$  with coefficient $r$ such that $rB $ and $B$ have the same center.  If $B$ is an Euclidean ball in $\mathbb{R}^n$ with center at $x$ and radius $r$, then $N(x,r)$ will denote the doubling index for this ball. 

 We will use the estimates of growth of harmonic functions in terms of the doubling index.
\begin{lemma} \label{ln3}
 For any $\varepsilon \in (0,1)$ there exist $C=C(\varepsilon,M,g,O)>0$ and $R=R(\varepsilon,M,g,O)>0$ such that 
\begin{equation} \label{n3}
t^{N(x,\rho)(1-\varepsilon)- C} \leq \frac{\sup\limits_{B(x,t\rho)} |u|}{\sup\limits_{B(x,\rho)} |u|}  \leq  t^{N(x,t\rho)(1+\varepsilon) + C}
\end{equation}   
 for any $x\in M$ and numbers $\rho>0$, $t>2$ satisfying $B(x,t\rho) \subset B(O,R)$ (and for any harmonic function $u$). Furthermore,
there exists $N_0=N_0(\varepsilon,M,g)$ such that if additionally $N(x,\rho)>N_0$, then 
\begin{equation} \label{n3*}
 t^{N(x,\rho)(1-\varepsilon)}  \leq \frac{ \sup\limits_{B(x,t\rho)} |u|}{\sup\limits_{B(x,\rho)} |u|}.
\end{equation}
\end{lemma}
% If $N(x,t\rho)>N_0$, then 
%\begin{equation} \label{n3**}
%\sup\limits_{B(x,t\rho)} |u|  \leq  t^{N(x,t\rho)(1+\varepsilon)}\sup\limits_{B(x,\rho)} |u|.
%\end{equation}

 The estimates \eqref{n3},\eqref{n3*} are corollaries from almost monotonicity of the frequency \eqref{b1} and standard elliptic estimates. For the convenience of the reader we deduce them in Lemma \ref{l4} and Lemma \ref{l5}. 

 We will show in Theorem \ref{thv} that there  exist $r=r(M)>0$ and $\alpha=\alpha(n)> 1$ such that the following inequality holds:
 $$H^{n-1}(\{u=0\}\cap B(O,r)) \leq C (N(O, Kr))^\alpha, $$
 where $K = K(n)\geq  2$  and $C=C(M)$.

Note that for real analytic manifolds one can replace $\alpha$ by $1$ in the estimate above, using complex analysis techniques (holomorhpic extension of a harmonic function to an open set in $\mathbb{C}^n$ and Jensen's formula  on one dimensional sections), see \cite{GL}. 

 We remark that only few properties of $H^{n-1}$  are used in the proof:
 subadditivity and the rescaling property. So there is a chance that the methods of this paper might be applied to other characteristics of nodal sets.
% By subadditivity we mean that if $S$,$Q_1$ and $Q_2$ are sets in $\mathbb{R}^n$, then
%$$H^{n-1}((Q_1\cup Q_2)\cap S) \leq H^{n-1}(Q_1\cap S) + H^{n-1}(Q_2\cap F).$$
%By the rescaling property we mean that if $S \subset \mathbb{R}^n$ and $a S$ is a homothetic image of $S$ with coefficient $a>0$, then   $$H^{n-1}(aS)= a^{n-1}H^{n-1}(S).$$

 We  outline the question we are trying to investigate in this paper:
Is the frequency additive in some sense? 

 Some partial positive answers are obtained in the simplex lemma and in the hyperplane lemma, which are combined to get the polynomial upper bounds for the volume of the nodal sets in terms of the frequency (or the doubling index). 

%%%%%%%%%%%%%%%%%%%%%%%%%%%%%%%%%%
\subsection*{Acknowledgments} 
 This work  was started in collaboration with Eugenia Malinnikova who suggested to 
 apply the combinatorial approach to nodal sets of Laplace eigenfunctions. Her role in this work is no less  than the author's one.  Unfortunately, she
refused to be a coauthor of this paper.
On various stages of this work  
I discussed
it with Lev Buhovsky and Mikhail Sodin. Eugenia, Lev and Mikhail also  
read the first
draft of this paper and made many suggestions and comments.  I thank all of them.

%  The earlier version of the paper contained a weaker estmate $H^{n-1}(\varphi_\lambda = 0) \geq c \lambda^{1/2 - \varepsilon}$ obtained by a different method, which didn't imply that a non-constant harmonic function in $\mathbb{R}^3$ has an infinite area of the nodal set. Michail Sodin suggested to treat the latter question using the general theory on the growth of power series of entire functions.  This suggestion was realised in the form of Lemma \ref{lem:mf} and helped the author to remove the $\varepsilon$. 

    This work was started while the author was visiting NTNU, continued at the Chebyshev Laboratory (SPBSU) and finished at TAU. The final version of the paper was completed at the Institute for Advanced Study. I am grateful to these institutions for their hospitatility and for great working conditions.
  
  The author was  supported in part by ERC Advanced Grant~692616 and ISF Grants~1380/13, 382/15 and by a Schmidt Fellowship at the Institute for Advanced Study.

\section{Simplex lemma} \label{sec:sl}
Let $x_1, \dots , x_{n+1}$ be vertices of a simplex $S$ in $\mathbb{R}^n$. The symbol $diam(S)$ will denote the diameter of $S$ and
 by $width(S)$ we will denote the width of $S$, i.e. the minimum distance between a pair of parallel hyperplanes such that $S$ is contained between them. Define the relative width of $S$: $w(S)= width(S) / diam(S)  $.
 Let $a>0$ and  assume that $w(S) >a$. In particular we assume that $x_1, \dots , x_{n+1}$ do not lie on the same hyperplane. For the purposes of the paper there will be sufficient a particular choice of $a$, which depends on the dimension $n$ only, the choice will be specified in Section \ref{sec:nq}.
Denote by $x_0$ the barycenter of $S$.

  We will use  an Euclidean geometry lemma:
 {\it there exist $c_1>0$, $K \geq 2/a$ depending on $a, n$ only  such that if $\rho = K diam(S)$, then
 $B(x_0,\rho(1+c_1)) \subset \cup_{i=1}^{n+1}B(x_i, \rho)$.} 

 We remark that if the simplex is very degenerate ($a$ is small), then $c_1$ has to be  small and the number $K$ has to be big: $$ c_1 \to 0, K \to +\infty  \textup{ as } a \to 0.$$

% Formulate lemma for harmonic functions first. After that  on manifold.
\begin{lemma} \label{sl}
 Let $B_i$ be balls with centers at $x_i$ and radii not greater than $\frac{K}{2} diam(S)$, $i=1, \dots, n+1$, where $K=K(a,n)$ is from the Euclidean geometry lemma.
 There exist positive numbers $c=c(a,n)$, $C=C(a,n)\geq K$, $r=r(M,g,O,a)$, $N_0=N_0(M,g,O,a)$ such that
 if $S \subset B(O,r)$ and
if $N(B_i)> N$ for each $x_i$, $i=1 \dots {n+1}$, where $N$ is a number greater than $N_0$,
 then $N(x_0, C diam(S))> N(1+c)$.

\end{lemma} 
\begin{proof}
 
In view of almost monotonicity of the doubling index \eqref{n3} we will assume that all $B_i$ have the same radius $\rho =K diam(S)$.

Let $M$ be the supremum of $|u|$ over the union of $B(x_i, \rho)$, then $|u|$ is not greater than $M$ in $B(x_0,\rho(1+c_1))$ and  $\sup \limits_{B(x_i, \rho)} |u| =M$ for some $i$.
 Let  $t>2$ and $\varepsilon>0$, these parameters will be specified later.   Assume that  \eqref{n3*} holds for $B(x_i, \rho t) $,  then $ \sup \limits_{B(x_i, \rho t) }|u| \geq M t^{N(1-\varepsilon)}$.

We need a metric geometry fact, which follows from the triangle inequality: {\it there exists $\delta=\delta(t) \in (0,1)$ such that $B(x_i, \rho t) \subset B(x_0, \rho t(1+\delta)) $ and $\delta(t) \to 0 $ as $t \to +\infty$. \it}

Let $\tilde N$ be the doubling index for $B(x_0, \rho t(1+\delta))$. Suppose \eqref{n3} holds for the pair of balls $B(x_0, \rho t(1+\delta))$ and $B(x_0, \rho(1+c_1))$, then
$$
\left[\frac{t(1+\delta)}{1+c_1}\right]^{\tilde N (1+\varepsilon) + C} \geq \frac{\sup\limits_{B(x_0, \rho t(1+\delta))}|u|}{\sup\limits_{B(x_0, \rho(1+c_1))}|u|} \geq  \frac{\sup\limits_{B(x_i, \rho t)}|u|}{\sup\limits_{B(x_0, \rho(1+c_1))}|u|} \geq \frac{M t^{N(1-\varepsilon)}}{M}= t^{N(1-\varepsilon)}.
$$
 Hence 
\begin{equation} \label{eq:sl1}
\left[\frac{t(1+\delta)}{1+c_1}\right]^{\tilde N (1+\varepsilon) + C} \geq t^{N(1-\varepsilon)}.
\end{equation}

  Now, we  specify our choice of parameters. We first choose $t>2$ so that $\delta(t)< c_1/2$, then 
\begin{equation} \label{eq:sl2}
\frac{t(1+\delta)}{1+c_1} \leq t^{1-c_2}
\end{equation}
 for some $c_2=c_2(t,c_1) \in (0,1)$. Second, we choose $\varepsilon=\varepsilon(c_2) >0$ and $c=c(c_2)>0$ such that
\begin{equation} 
 \frac{1-\varepsilon}{(1+\varepsilon)(1-c_2)}> 1+2c.
\end{equation}
 Third, we choose 
$R=R(\varepsilon,M,g,O) >0$ and $N_0=N_0(\varepsilon,M,g,O)$ such that Lemma \ref{ln3} holds for these parameters and put $r:=R/(10Kt)$. This choice of $r$ provides \eqref{n3} for the pair of balls $B(x_0, \rho t(1+\delta))$ and $B(x_0, \rho(1+c_1))$ and  \eqref{n3*} for $B(x_i, \rho t)$. Hence
the inequality \eqref{eq:sl1} holds and \eqref{eq:sl2} gives $$t^{(1-c_2)(\tilde N (1+\varepsilon) +C)} \geq t^{N(1-\varepsilon)}.$$  We therefore have $$\tilde N \geq N \frac{(1-\varepsilon)}{(1+\varepsilon)(1-c_2)} - C_1 \geq N(1+2c) - C_1 \geq N(1+c) + (c N_0 - C_1).$$ We can also ask $N_0$ to be big enough so that $c N_0 - C_1>0$. Thus $$\tilde N > N(1+c).$$

\end{proof}

\section{Propagation of smallness of the Cauchy data}
 
 If one considers  a smooth Riemannian metric $g$ in a unit cube $Q$  in $\mathbb{R}^n$, then any harmonic function $u$ (with respect to $g$) 
 satisfies $Lu=0$, where $L$  is a uniformly elliptic (in a slightly smaller cube) operator of second order in the divergence form with smooth coefficients.  Consider a cube $q \subset \frac{1}{2}Q$ with  side $r$
 and let $F$ be a face of $ q$.
 In this section we formulate a result that we will refer to as the propagation of smallness of the Cauchy data for elliptic PDE. See Lemma 4.3 in \cite{L} and Theorem 1.7 in \cite{Cauchy} for the proof of the result below, which we bring not in full generality but in a convenient way for our purposes.

{\it Suppose that $|u|\leq 1$ in $q$. There exist $C>0$ and $\alpha \in(0,1)$, depending on $L $ only  such that  if $|u|\leq \varepsilon$ on $F$ and $|\nabla u| \leq \frac{\varepsilon}{r}$ on $F$, $\varepsilon<1$, then } 
\begin{equation} \label{eq:ps}
\sup\limits_{\frac{1}{2}q}|u| \leq C \varepsilon^{\alpha}.
\end{equation}
\begin{remark}
 We will apply propagation of smallness of the Cauchy data in the case when the coefficients of the operator $L$ are sufficiently close in the $L^\infty$ norm to the coefficients of the
 standard Laplace operator $\Delta$ in $B(O,R_0)$ and the derivatives of coefficients $L$ are sufficiently small. Under these assumptions $\alpha$ can be chosen to depend only on $n$, see Theorem 1.7 in \cite{Cauchy}.
\end{remark}
\section{Hyperplane lemma}

 Given a cube $Q$, we will denote $\sup\limits_{x \in Q, r \in (0,diam(Q))} N(x,r)$ by $N(Q)$ and call it the doubling index of $Q$. This definition is  different than a doubling index for balls but more convenient in the following sense. If a cube $q$ is contained in a cube $Q$, then $N(q)\leq N(Q)$. Furthermore if a cube $q$ is covered by  cubes $Q_i$ with $diam(Q_i) \geq diam(q)$, then $N(Q_i)\geq N(q)$ for some $Q_i$.  

\begin{lemma} \label{lp1}
 Let $Q$ be a cube $[-R, R]^n $ in $\mathbb{R}^n$.
 Divide $Q$ into $(2A+1)^n$  equal subcubes $q_i$ with side-length $\frac{2R}{2A+1}$. 
 Consider the cubes $q_{i,0}$ that have non-empty intersection with the hyperplane $x_n=0$.
Suppose that for each  $q_{i,0}$ there exists $x_i \in q_{i,0}$ and $r_i < 10 diam(q_{i,0}) $ such that $N(x_i,r_i) > N$, where $N$ is a given positive number. 
Then there exist  $A_0=A_0(n)$, $R_0=R_0(M,g,O)$, $N_0 =N_0(M,g,O)$  such that if $A>A_0$, $N> N_0$, $R<R_0$ then $N(Q)> 2 N $.
\end{lemma}
\begin{proof}
  We will ask $R_0$ to be small enough so that  Lemma \ref{ln3} holds with $\varepsilon =1/2$ and $10 n\cdot  R_0$ in place of $R$ in Lemma \ref{ln3}. Also  we may assume that coefficients of $L$ are close to the coefficients of the standard Laplacian in $C^1(B(O, 10 n\cdot  R_0))$ to be able to use \eqref{eq:ps}. We have described our choice of $R_0$.
 
For the sake of simplicity we will assume that $R=1/2$ and  $R_0 \geq 1/2$. The general case follows by changing the scale in the argument below.

Let $B$ be the unit ball $B(O,1)$ and
let $M$ be the supremum of $u$ over $\frac{1}{8} B$.
 For each $x_i$ in $\frac{1}{16} B$ the ball $B(x_i, 1/32)$ is contained in $1/8 B$. Hence $\sup \limits_{B(x_i, 1/32)}|u| \leq M$. Using $N(x_i, r_i) > N$ and \eqref{n3} with $\varepsilon =1/2$ we get 
$$\sup \limits_{2q_{i,0}}|u| \leq \sup \limits_{B(x_i,\frac{4\sqrt n}{2A+1})}|u| \leq  C \sup\limits_{B(x_i, 1/32)}|u|  \left(\frac{128 \sqrt n}{2A+1}\right)^{\frac{N}{2}} \leq  M 2^{-c N \log A},$$ where  
$c=c(n)>0$. In the last inequality we assumed that $A > A_0 (n)$ and $N$ is sufficiently large .

 By a standard elliptic estimate  
$$\sup \limits_{q_{i,0}}| \nabla u| \leq C A \sup \limits_{2q_{i,0}}|u| \leq C A M  2^{-cN  \log A}\leq   M 2^{-c_1(n)N  \log A}.$$
Thus $|u|$ and $|\nabla u|$ are bounded by $ M 2^{-c_1N  \log A}$ on $\frac{1}{8} B \cap \{x_n=0\}$.

 Let $q$ be a cube with side $\frac{1}{16 \sqrt n }$  in the halfspace $\{x_n>0 \}$ such that $q \subset \frac{1}{8} B$ and $$  \frac{1}{32 \sqrt n} B \cap \{x_n=0\} \subset \partial q \cap \{x_n=0\}.$$ In other words, $q$ has a face $F$ on the hyperplane $\{x_n=0\}$. Let $p$ be the center of $q$, then $B(p,\frac{1}{32\sqrt n}) \subset q $. 

 Consider the function $v= \frac{u}{  M}$, which absolute value is not greater than $1$ in $q$. The Cauchy data of $v$ is small on $F$: $|v|$ and $|\nabla v|$ are smaller than $  2^{-c_1N  \log A}$. Denote $  2^{-c_1N  \log A}$ by $\varepsilon$. 
Applying propagation of smallness for the Cauchy data, we obtain $\sup\limits_{\frac{1}{2}q}|v| \leq \varepsilon^\alpha $. In terms of $u$ we have $\sup\limits_{\frac{1}{2}q}|u| \leq M   \varepsilon^\alpha = M 2^{-\alpha c_1 N  \log A }$. 

 The ball $B(p,\frac{1}{64\sqrt n})$ is contained in $\frac{1}{2}q$ and therefore $$\sup\limits_{B(p,\frac{1}{64\sqrt n})}|u| \leq M 2^{-\frac{\alpha c N}{2}  \log A }.$$ 
 However $\sup\limits_{B(p,1/2)}|u| \geq M$ since $\frac{1}{8} B \subset B(p,1/2)$. Hence
 $$ \frac{\sup\limits_{B(p,1/2)}|u|}{\sup\limits_{B(p,\frac{1}{64\sqrt n})}|u|} \geq 2^{ \alpha c_1 N  \log A }.$$
 Denote by $\tilde N$ the doubling index for $B(p,1/2)$. By \eqref{n3} with $\varepsilon=1/2$ we have 
$$ \frac{\sup\limits_{B(p,1/2)}|u|}{\sup\limits_{B(p,\frac{1}{64\sqrt n})}|u|} \leq (64\sqrt n)^{\tilde N/ 2}.$$ Hence $\tilde N \geq c_2 N \log A$ for some $c_2=c_2(n)>0$, and $\tilde N \geq 2N$ for $A$ big enough.
 \end{proof}
\begin{corollary} \label{lp2} Let $Q$ be a cube $[-R, R]^n $ in $\mathbb{R}^n$ and $N(Q) $ is not greater than a number $ N$. For any $\varepsilon >0$ there exists an odd positive integer $A_1=A_1(n,\varepsilon)$ such that the following holds.
Let us divide $Q$ into $A_1^n$ smaller equal subcubes $q_i$ and
 consider the cubes $q_{i,0}$ that have non-empty intersection with the hyperplane $x_n=0$.
If  $N> N_0(M,g,O)$, $R<R_0(M,g,O)$, then
 the number of subcubes $ q_{i,0}$ that have doubling index  greater than $N/2$  is less than $\varepsilon  A_1^{n-1}$.
\end{corollary}
\begin{proof}
 According to Lemma \ref{lp1} we can choose an integer $A_0$ and $N_0>0$, assume $N>N_0$, 
 partition $Q$ into $(2A_0+1)^n$ equal subcubes, and then  at least one subcube with non-empty intersection with $\{ x_n=0\}$ has doubling index smaller than $N/2$.  
 
 Now, let us partition $Q$ into $(2A_0+1)^{kn}$ equal subcubes $q_i$ and denote by $M_k$ the number of subcubes  with non-empty intersection with $\{ x_n=0\}$ and doubling index greater than $N/2$. If a cube $q_i$ has doubling index smaller than $N/2$, then any its subcube also does. 

It is not important in the proof of  Lemma \ref{lp1} that $Q$ is a cube with center at the origin, the same argument shows that if we divide  a cube $q_i$, which has non-zero intersection with $\{ x_n=0\}$,  into  $(2A_0+1)^n$ equal subcubes, then at least one subcube  with non-empty intersection with $\{ x_n=0\}$ has doubling index smaller than $N/2$. This observation gives $M_{k+1} \leq M_k ((2A_0+1)^{n-1} -1)$.

 Thus  $M_k \leq (1-\frac{1}{(2A_0+1)^{n-1}})^k (2A_0+1)^{k(n-1)}$. Choosing $k$ so that $(1-\frac{1}{(2A_0+1)^{n-1}})^k \leq \varepsilon$, we finish the proof.
\end{proof}
 \begin{remark} 
 The same argument shows that  in Lemma \ref{lp1} and in Corollary \ref{lp2} one can replace  $Q$ by any its homothety-rotation-shift copy $ Q_r \subset B(O,R_0)$  , $r \in (0,1)$, $R_0=R_0(M,g,O)$ and replace the hyperplane $\{ x_n = 0\} $ by a hyperplane that contains the center of $Q_r$ and is parallel to one of its faces,  Lemma \ref{lp1} and Corollary \ref{lp2} will remain true with $A_1$,$A_0$ and $N_0$ independent of $r$.
 \end{remark}

\section{Number of cubes with big doubling index} \label{sec:nq}
In this section we follow notation from Sections \ref{sec:di} and \ref{sec:sl}. 
The next theorem seems to be a useful tool in nodal geometry. We will apply it later to obtain upper estimates of the volume of the nodal sets in terms of the doubling index. 
\begin{theorem} \label{th1/2}
  There exist   constants $c>0$, an integer $A$ depending on the dimension $d$ only and positive numbers $N_0=N_0(M,g,O)$, $r=r(M,g,O)$ such that for any cube $Q\subset B(O,r)$ the following holds: if we partition $Q$ into $A^n$ equal subcubes, then the number of subcubes with doubling index greater than $\max(N(Q)/(1+c),N_0)$ is less than $\frac{1}{2} A^{n-1}$.
 
\end{theorem}
\begin{proof}
 Let us fix a small $\varepsilon>0$, which will be specified later, and choose $A_1=A_1(\varepsilon,n)$ such that Corollary \ref{lp2} holds for this $\varepsilon$ and $A_1=2A_0+1$  as well as the remark after Corollary \ref{lp2}.  Let us subsequently divide $Q$ into equal subcubes so that at $j$-th division step $Q$ is partitioned  into $(2A_0+1)^{nj}$ equal subcubes $Q_{i_1,i_2, \dots, i_j}$, $i_1,i_2\dots, i_j \in \{1,2,\dots, (2A_0+1)^n\}$, so that $Q_{i_1,i_2, \dots, i_j} \subset Q_{i_1,i_2, \dots, i_{j-1}}$.
Let the parameter $c>0$. We will say that the  cube $Q_{i_1,i_2, \dots, i_j}$ is bad if $N(Q_{i_1,i_2, \dots, i_j})> N(Q)/(1+c)$ and good otherwise. 
  
   Fix a cube $Q_{i_1,i_2, \dots, i_j}=:q$, we are interested in the number of its bad subcubes $Q_{i_1,i_2, \dots, i_{j+1}}=:q_{i_{j+1}}$. For the sake of convenience we will omit index $j+1$ and write $q_{i}$ in place of $q_{i_{j+1}}$.  We will prove the following lemma.
\begin{lemma}  \label{ldi1}
If $\varepsilon,c$ are sufficiently small,  and $j>j_0$, where $j_0=j_0(\varepsilon,c)$,  then $\# \{i: N(q_i)> N(Q)/(1+c)\} \leq \frac{1}{2} (2A_0+1)^{n-1}$ 
\end{lemma}
 
Let $F$ be the set of all points $x$ in $q$ such that there exists $r\in (0,diam(q_i)]$ such that $N(x,r) > N(Q)/(1+c)$. 
 If a closed cube $q_i$ is bad, then it contains at least one point from $F$.
 We use the notation $\tilde w(F):= \frac{width(F)}{diam(q)}$ for the relative width of $F$ in $q$. 
 We will prove  Lemma \ref{ldi1} after the following lemma.
 
 \begin{lemma} \label{lw}
For any $w_0>0$ there exist a positive integer $j_0$  and a constant $c_0>0$ such that if 
 $j>j_0$, $c<c_0$, then $\tilde w(F)< w_0$.
\end{lemma}
To prove this lemma we need an Euclidean geometry fact: {\it for any set of points $F$ in $q$ with non-zero $\tilde w(F)$  there exists 
 $a=a(\tilde w(F),n)>0$ and a simplex  $S \subset F$ such that $ w(S)> a$ and $diam(S) > a\cdot diam(q)$}. 

 For each vertex $x_k$ of $S$ there is a ball $B(x_k,r_k)$ with $N(x_k,r_k) \geq N/(1+c)$ and $r_k \leq diam(q) \leq \frac{1}{a} diam(S)$.
 We can apply Lemma \ref{sl} for the simplex $S$. Then $N(x_0,C_0 diam(S))> (1+c_0)N/(1+c)$, where $x_0$ is a barycenter of $S$ and $c_0$,$C_0$ are positive constants  depending on $a$ (and $n$) only and therefore on $\tilde w(F)$ only (and $n$). If $c_0> c$ and $C_0 diam(S) \leq diam(Q)$ that means a contradiction with $N(Q) \leq N$. This is why we require $j$ to be big enough:
$diam(S) \leq diam(q) \leq \frac{diam(Q)}{(2A_0+1)^{j}}\leq diam(Q)/3^{j}$ and it is sufficient to take $j$ such that $3^{j}>C_0$.

 Now,  Lemma \ref{lw}  is proved and we can think that $\tilde w(F)$ is smaller than a fixed number $w_0=\frac{1}{2A_0+1}$ and proceed to prove Lemma \ref{ldi1}.
 There exists a hyperplane $P$ such that its $w_0\cdot  diam(q)$ neighborhood  contains all $F$.
 Furthermore, we  can find a biggier cube $\tilde q$ with one face parallel to $P$ such that the center of $\tilde q $ is  in $ P\cap q $ and
 $diam(\tilde q) =10 \sqrt n \cdot diam(q)$. Automatically $\tilde q$ contains $q$. Divide $\tilde q$ into $(2A_0+1)^n$ equal subcubes $\tilde q_i$. We will denote by $\tilde q_{i,0}$ such subcubes that have non-zero intersection with $P$.  Since $w_0\leq \frac{1}{2A_0+1}$, each bad cube $q_i$ is contained in a $\frac{2\sqrt n \cdot diam(q)}{2A_0+1}$ neighborhood of $P$ and  each bad cube $q_i$ is covered by a finite number (which depends on $n$ only) of $\tilde q_{i,0}$. Therefore the number of bad cubes $q_i$ is less than the number of bad cubes $\tilde q_{i,0}$ times some  constant depending on dimension $n$ only.
 
  Now, assume the contrary to Lemma \ref{ldi1}. Suppose that the number of bad $q_i$ is greater than $\frac{1}{2} (2A_0+1)^{n-1}$, then the number of bad cubes $\tilde q_{i,0}$ is at least  $\frac{1}{C} (2A_0+1)^{n-1}$, where $C=C(n)>0$.

Finally, we choose $\varepsilon$, which didn't play a role till now: $\varepsilon$ is any number in $(0,\frac{1}{2C})$. Recall that $A_0=A_0(\varepsilon)$  is such that Corollary \ref{lp2}  holds for $A_1=2A_0+1$ and this $\varepsilon$ as well as the remark after Corollary \ref{lp2}. Since the number of bad $\tilde q_{i,0}$ is greater than $\varepsilon(2A_0+1)^{n-1}$ we have $N(\tilde q) \geq 2N/(1+c)$. Without loss of generality we assume that $c<1/10$, then
 there exists a point $\tilde p\in \tilde q$ such that $N(\tilde p,diam (\tilde q)) \geq \frac{3}{2} N $. The last observation looks to be inconsistent with $N(Q) \leq N$, however $\tilde q$ is not necessarily contained in $Q$ and the contradiction is not immediate. 
 This obstacle is easy to overcome. Consider any point $p \in q \subset Q$. There exists a large $C_1=C_1(n)$ such that $N(p, C_1 diam(\tilde q) ) \geq (1-1/100) N(\tilde p, diam(\tilde q))$ (see Lemma \ref{l6}). 
Thus there is a contradiction with $N(Q) \leq N$ since $N(p, C_1 diam(\tilde q)) > N$ and  $C_1 diam(\tilde q) \leq  diam(Q)$ if $j$ is big enough.  The proof of Lemma \ref{ldi1} is completed. Now, it is a straightforward matter to prove Theorem \ref{th1/2}.

 Denote by $K_j$ the number of bad cubes on $j$-th step.  If $Q_{i_1,i_2, \dots, i_j}$ is good, then any its subcube is also good by the definition of doubling index for cubes. If $Q_{i_1,i_2, \dots, i_j}=:q$ is bad, then by  Lemma \ref{ldi1} the number of bad subscubes $Q_{i_1,i_2, \dots, i_{j+1}}$ in $q$ is less than $\frac{1}{2}(2A_0+1)^{n-1}$. Hence  $K_{j+1} \leq \frac{1}{2}(2A_0+1)^{n-1}K_j$ for $j>j_0$. We can define $A=(2A_0+1)^{j}$ and see that $K_j \leq  K_{j_0} \frac{1}{2^{j-j_0}}(2A_0+1)^{(n-1)(j-j_0)} \leq \frac{1}{2} A^{n-1}  $ for $j$ big enough.

\end{proof}
%%%%%%%%%%%%%%%%%%%%%

\section{Upper estimates of the volume of the nodal set.}

\begin{theorem} \label{thv}
 There exist positive numbers $r=r(M,g,O)$, $C=C(M,g,O)$ and $\alpha=\alpha(n)$ such that for any harmonic function $u$ on $M$ and any cube $Q\subset B(O,r)$

\begin{equation} \label{thv1}
 H^{n-1}(\{u=0\}\cap Q) \leq C diam^{n-1}(Q) N_u^\alpha(Q),
\end{equation}
 where $N_u(Q)$ is the doubling index of $Q$ for the function $u$.
\end{theorem}
\begin{proof}
 Choose $r$ so that Theorem \ref{th1/2} holds with this $r$ and some $c=c(n)$, $A=A(n)$.
Now, define the function 
$$F(N):= \sup  \frac{H^{n-1}(\{u=0\}\cap Q)}{ diam^{n-1}(Q)},$$ 
where the supremum  is taken over the set of harmonic functions $u$ on $M$, which we denote by $\textup{Harm}(M)$, and cubes $Q$ within  $B(O,r)$ such that $N_u(Q) \leq N$.
 The estimate \eqref{thv1} is equivalent to 
\begin{equation} \label{thv2}
 F(N) \leq C N^{\alpha}.
\end{equation} 
 We note that if $u$ changes a sign in $Q$, then $N_u(Q) \geq 1$, since $\lim\limits_{t \to +0} N(x,t)$ is equal to the vanishing order of $u$ at $x$.
  Due to the Hardt-Simon exponential bounds we know $F(N)<+\infty$ for each positive $N$. 
 
 We will call $N>0$ bad if 
\begin{equation} \label{thv3}
F(N) > 4A \cdot F(N/(1+c)) .
\end{equation}
 Our goal is to show that the set of bad $N$ is bounded.  In view of monotonicity of $F$ it  would imply  \eqref{thv2} immediately, where the constant $\alpha$ depends  on $A$ and $c$ only and therefore only on the dimension $n$.

 Consider a bad $N$ and a function $u$ with a cube $Q$ such that $F(N)$ is almost attained for them: 
\begin{equation} \label{thv4}
\frac{H^{n-1}(\{u=0\}\cap Q)}{ diam^{n-1}(Q)} > \frac{3}{4} F(N)
\end{equation}
while $N_u(Q) \leq N$. Divide $Q$ into $A^{n}$ equal subcubes $Q_i$, $i=1, \dots, A^n$. Divide $Q_i$ into two groups  $G_1:=\{Q_i: N/(1+c)< N(Q_i) \leq N\}$ and $G_2:=\{Q_i: N(Q_i)\leq N/(1+c)\}$. By Theorem \ref{th1/2} we know that the number of cubes in $G_1$ satisfies $|G_1| \leq  \frac{1}{2} A^{n-1}$
if $N>N_0(M,g)$. 
Note that  $$H^{n-1}(\{u=0\}\cap Q) \leq \sum\limits_{Q_i \in G_1} H^{n-1}(\{u=0\}\cap Q_i) +  \sum\limits_{Q_i \in G_2} H^{n-1}(\{u=0\}\cap Q_i) $$
$$\leq |G_1| F(N) \frac{diam^{n-1}(Q)}{A^{n-1}} + |G_2| F(N/(1+c)) \frac{diam^{n-1}(Q)}{A^{n-1}} = I +II.$$
Since $|G_1| \leq  \frac{1}{2} A^{n-1}$ we can estimate $I \leq \frac{1}{2} F(N) diam^{n-1}(Q)$. Using that $N$ is bad, we have $II \leq |G_2| \frac{F(N)}{4A} \frac{diam^{n-1}(Q)}{A^{n-1}}  $ and $|G_2| \leq A^n$, hence $II \leq \frac{1}{4} F(N) diam^{n-1}(Q) $. Finally, $ H^{n-1}(\{u=0\}\cap Q) \leq \frac{3}{4} F(N) diam^{n-1}(Q)$ and the last inequality contradicts to \eqref{thv4}. Thus we had shown that the set of bad $N$ is bounded by some $N_0=N_0(M,g)$.
\end{proof}

\begin{theorem}
 Let $(W,g)$ be a compact $C^{\infty}$-smooth Riemannian manifold without boundary. For a Laplace eigenfunction $\varphi$ on $W$ with $\Delta \varphi + \lambda \varphi=0$ define its nodal set $Z_\varphi:=\{ \varphi =0 \}$. There exist $C=C(W,g)$ and $\alpha$, depending only on the dimension $n$ of $W$, such that 

 $$ H^{n-1}(Z_\varphi) \leq C \lambda^{\alpha}.$$
\end{theorem}
 \begin{proof}
 We will use a standard trick that allows to pass  from Laplace eigenfunctions to harmonic functions by adding an extra variable.
 Consider a product manifold $ M = W \times \mathbb{R}$, where one can define a harmonic function $u$ by
$$u(x,t)= \varphi(x) e^{\sqrt \lambda \cdot t}, x\in W, t \in \mathbb{R}.$$   
 
 The Donnelly-Fefferman doubling index estimate for Laplace eigenfunctions claims 
 $$\sup \limits_{B_g(p,2r)} |\varphi| \leq 2^{C \sqrt \lambda}  \sup \limits_{B_g(p,r)} |\varphi|,$$
where $C=C(M,g)$, $p$ is any point on $W$ and $r \in (0, R_0(M,g))$.
It implies that the doubling index of $u$ is also bounded by $C_1 \sqrt \lambda$ in balls with radius smaller than some $R_1=R_1(W,g)$.
 Let us fix a point $O \in M$ and a point $\tilde O=(O,0) \in  M$.
  We can apply Theorem \ref{thv}  to see that $H^{n}(\{u=0\} \cap B(\tilde O, r) ) \leq C_2 \lambda^{\alpha}$ for some $r=r(W,g)>0$. 
%We remark that in Theorem \ref{thv} there was a requirement for doubling index in cubes in different metric (metric coming from normal coordinates) to $g$, but by the equivalence of metrics \eqref{eqm} the doubling index in new metric can be bounded by $C_2 \sqrt \lambda$ ( for balls in some neighborhood of $\tilde O$).

  It remains to note that $H^{n}(\{u=0\} \cap B(\tilde O, r) ) \leq C_2 \lambda^{\alpha}$ implies $H^{n-1}(\{\varphi =0\} \cap B_g(O, r/2) ) \leq  C_3 \lambda^{\alpha}$ since the zero set of $u$ is exactly $Z_\varphi \times \mathbb{R}$.
 Finally, one can cover $M$ by finite number of such  balls  and obtain the desired global estimate of the volume of the nodal set. 
 \end{proof}

 \textbf{Remark.}
  The same argument gives a local volume estimate of the nodal set:
 $$H^{n-1}(\{\varphi =0\}\cap B_g(O,r)) \leq C r^{n-1} \lambda^{\alpha}.$$

%%%%%%%%%%%%%%%%%%%%%%%%%%%%%%%%%
\section{Auxiliary lemmas }
\begin{lemma} If $\varepsilon_1>0$ is a sufficiently small number ($\varepsilon_1 < 1/10^{10}$), then there exist $C= C(\varepsilon_1, M,g,O)>0$ and $R_1= R_1(\varepsilon_1, M,g,O)>0$ such that
\begin{equation} \label{b2}
\beta(p,2r(1+\varepsilon_1))(1+ 100 \varepsilon_1 ) + C \geq N(p,r) \geq  \beta(p,r(1+\varepsilon_1))(1- 100 \varepsilon_1) - C
\end{equation} 
for $r \in (0,R_1)$ and $p \in B(O,R_1)$.
\end{lemma}
  We remark that it is not a misprint and the argument in $\beta$  in the right-hand side of \eqref{b2} is strictly greater than $r$.
\begin{proof}
 % By the definition $N(r) = \log_2 \frac{\sup\limits_{B_{2r}} |u|}{\sup\limits_{B_{r}} |u|}$. 
 By the equivalence of metrics \eqref{eqm} we have
$B(p,r) \subset B_g(p,r(1+\varepsilon)) $ and by the standard elliptic estimate  $$\sup\limits_{B(p,r)} |u|^2 \leq\sup\limits_{B_g(p,r(1+\varepsilon))} |u|^2 \leq C_1 H(r(1+\varepsilon)^2)/ r^{n-1}$$ and $$\sup\limits_{B(p,2r)} |u|^2 \geq \sup\limits_{B_g(p,2r(1-\varepsilon))} |u|^2 \geq C_2 H(2r(1-\varepsilon))/r^{n-1}.$$ 
Hence  we can estimate $$N(p,r) = \frac{1}{2}\log_2 \frac{\sup\limits_{B(p,2r)} |u|^2}{\sup\limits_{B(p,r)} |u|^2} \geq \frac{1}{2}\log_2 \frac{1}{C_3} \frac{H(2r(1-\varepsilon))}{ H(r(1+\varepsilon)^2)}, $$ and by  Corollary \ref{logc} the right-hand side is at least $$ \log_2 \left[ \frac{1}{C_3} \left(\frac{2(1-\varepsilon)}{(1+\varepsilon)^2}\right)^{\beta(r(1+\varepsilon)^2) / (1+\varepsilon)} \right] \geq  \beta(r(1+\varepsilon)^2)(1- 20 \varepsilon) - C_4.$$  We assumed above that $\varepsilon$ is sufficiently small.
 Now, we can let $\varepsilon_1 $ be such that $(1+\varepsilon)^2 = 1+\varepsilon_1 $, so $\varepsilon_1 \sim 2 \varepsilon$, and the right-hand side inequality of \eqref{b2} is obtained.

 To obtain the opposite estimate we argue in the same manner: 
$$\sup\limits_{B(p,r)} |u|^2 \geq\sup\limits_{B_g(p,r(1-\varepsilon))} |u|^2 \geq C_2H(r(1-\varepsilon))/ r^{n-1},$$ 
 and  $$\sup\limits_{B(p,2r)} |u|^2 \leq \sup\limits_{B_g(p,2r(1+\varepsilon))} |u|^2 \leq C_3 H(2r(1+\varepsilon)^2)/r^{n-1}.$$
 Applying these estimates we have
$$N(p,r) = \log_2 \frac{\sup\limits_{B(p,2r)} |u|}{\sup\limits_{B(p,r)} |u|} \leq \frac{1}{2} \log_2 C_4 \frac{H(2r(1+\varepsilon)^2)}{H(r(1-\varepsilon))} .$$ In a view of \eqref{logc}, the right hand side can be estimated from above by $$\beta(2r(1+\varepsilon)^2)(1+ 20 \varepsilon ) + C_5 \leq \beta(p,2r(1+\varepsilon_1))(1+ 100 \varepsilon_1 ) + C_5, $$  where  $\varepsilon_1 $ satisfies $(1+\varepsilon)^2 = 1+\varepsilon_1 $.
\end{proof}
\begin{lemma} \label{l4}
 Let $\varepsilon$ be a small positive number. Then there exists $R= R(\varepsilon,M,g,O)$ such that for any $x \in B(O,R)$ and for any numbers $t>2$ and $\rho>0$ such that $t \rho < R$, 
\begin{equation} \label{n1}
  \sup\limits_{B(x,t\rho)} |u| \geq  t^{N(x,\rho)(1-\varepsilon) - C(\varepsilon,M,g)} \sup\limits_{B(x,\rho)} |u|.
\end{equation}
Furthermore, there exists $N_0=N_0(\varepsilon,M,g)$ such that  if $N(x,\rho) > N_0$, then   additionally
\begin{equation} \label{n1*}
  \sup\limits_{B(x,t\rho)} |u| \geq  t^{N(x,\rho)(1-\varepsilon) } \sup\limits_{B(x,\rho)} |u|.
\end{equation}
\end{lemma}
\begin{proof}
 We can assume that $t> 2^{1+\varepsilon}$, otherwise $t^{N(x,\rho)(1-\varepsilon)} \leq 2^{N(x,\rho)}$ and $$\sup\limits_{B(x,t\rho)} |u| \geq \sup\limits_{B(x,2\rho)} |u| \geq 2^{N(x,\rho)}  \sup\limits_{B(x,\rho)} |u| \geq  t^{N(x,\rho)(1-\varepsilon) } \sup\limits_{B(x,\rho)} |u|.$$

 Hereafter the constants $C_1$, $C_2$, \dots will be positive numbers depending on $\varepsilon$, $M$, $g$ only. 
 The inequality \eqref{h2} says that 
\begin{equation} \label{eq:l41}
\sup \limits_{B(x,t\rho)}|u|^2 \geq C_6 \frac{H(x,t\rho)}{(t\rho)^{n-1}}.
\end{equation}

 Let $\varepsilon_1$ be equal to $\varepsilon/1000$. We can apply \eqref{b2} for $\varepsilon_1$ to see that $$\beta(x,2\rho(1+\varepsilon_1))(1+100\varepsilon_1) + C_7 \geq N(x,\rho).$$ 
In view of  Corollary \ref{logc}   we obtain 
\begin{equation}  \label{eq:l42}
H(x,t\rho) \geq H(x,2\rho(1+\varepsilon_1)) \left(\frac{t}{2(1+\varepsilon_1)}\right)^{\frac{2N(x,\rho)}{(1+100\varepsilon_1)(1+\varepsilon_1)} - C_8}.
\end{equation}
 We use $t> 2^{1+\varepsilon}$ to ensure that $t\rho> 2\rho(1+\varepsilon_1)$.
 A standard elliptic estimate yields 
 \begin{equation}  \label{eq:l43}
H(x,2\rho(1+\varepsilon_1))\geq C_9 \rho^{n-1} \sup \limits_{B(x,2\rho)}|u|^2= C_9 2^{2N(x,\rho)} \rho^{n-1} \sup \limits_{B(x,\rho)}|u|^2.
 \end{equation}
Combination of  \eqref{eq:l41},\eqref{eq:l42},\eqref{eq:l43} implies
$$\sup \limits_{B(x,t\rho)}|u| \geq C_{10} 2^{N(x,\rho)} t^{-(n-1)/2} \left(\frac{t}{2(1+\varepsilon_1)}\right)^{N(x,\rho)/(1+200\varepsilon_1) - C_8/2} \sup \limits_{B(x,\rho)}|u|  $$
$$\geq C_{10}t^{-(n-1)/2}  \left(\frac{t}{(1+\varepsilon_1)}\right)^{N(x,\rho)/(1+200\varepsilon_1) - C_{11}} \sup \limits_{B(x,\rho)}|u|.$$
 Now, in order to establish \eqref{n1} it is sufficient to note that $$t^{N(x,\rho)/(1+200\varepsilon_1)} \geq  t^{N(x,\rho)(1-\varepsilon/2)} \geq t^{N(x,\rho)(1-\varepsilon)} 2^{N(x,\rho)\varepsilon/2}$$ and $$2^{N(x,\rho)\varepsilon/2} \geq 
(1+\varepsilon_1)^{N(x,p)/(1+200\varepsilon_1)}$$ since $2^{\varepsilon/2} \geq 1+\varepsilon/100$.

 The inequality \eqref{n1*} follows immediately   from \eqref{n1} if we apply it to twice smaller $\varepsilon$, require
$N(x,\rho) > \frac{2}{\varepsilon} C(\varepsilon/2,M,g)$ and put a  new smaller $R= R(\varepsilon/2,M,g)$.
\end{proof}
%%%%%%%%%%%%%%%%%%%%%%%%%%%%%%%%%%%%%%%%%%%%%%%

\begin{lemma} \label{l5}
 Let $\varepsilon$ be a small positive number. Then there exists $R= R(\varepsilon,M,g,O)$ such that for any $x \in B(O,R)$ and any numbers $t>2$ and $\rho>0$ such that $t \rho < R$ 
\begin{equation} \label{n2}
  \sup\limits_{B(x,t\rho)} |u| \leq  t^{N(x,t\rho)(1+\varepsilon) + C(\varepsilon,M,g)} \sup\limits_{B(x,\rho)} |u|.
\end{equation}
Furthermore, there exists $N_0=N_0(\varepsilon,M,g,O)$ such that  if $N(x,\rho) > N_0$, then   additionally
\begin{equation} \label{n2*}
  \sup\limits_{B(x,t\rho)} |u| \leq  t^{N(x,t\rho)(1+\varepsilon) } \sup\limits_{B(x,\rho)} |u|.
\end{equation}
\end{lemma}
\begin{proof} The proof is parallel to the proof of the previous lemma. Put $\varepsilon_1=\varepsilon/1000$.
 
 Inequality \eqref{h1} says that 
\begin{equation} \label{eq:l51}
\sup \limits_{B(x,t\rho)}|u|^2 \leq C_1 \frac{H(x,t\rho(1+\varepsilon_1))}{(t\rho)^{n-1}}.
\end{equation}

 We can apply \eqref{b2} for $\varepsilon_1$ to see that $\beta(x,t\rho(1+\varepsilon_1))\leq N(x,t\rho)(1+100\varepsilon_1)+C_2$. In view of the corollary \eqref{logc}  we obtain 
\begin{equation}  \label{eq:l52}
H(x,t\rho(1+\varepsilon_1)) \leq H(x,\rho) \left(t(1+\varepsilon_1)\right)^{2N(x,t\rho)(1+100\varepsilon_1)(1+\varepsilon_1) + C_3}.
\end{equation}
  Inequality \eqref{h2} implies
 \begin{equation}  \label{eq:l53}
 H(x,\rho) \leq C_5  \sup\limits_{B(x,\rho)} |u|^2 \rho^{n-1}.
 \end{equation}
Combination of  \eqref{eq:l51},\eqref{eq:l52},\eqref{eq:l53} gives us
$$\sup \limits_{B(x,t\rho)}|u| \leq C_{6} \frac{ \left(t(1+\varepsilon_1)\right)^{N(x,t\rho)(1+100\varepsilon_1)(1+\varepsilon_1) + C_3}}{t^{(n-1)/2}} \sup \limits_{B(x,\rho)}|u|  $$
$$\leq C_{7} \left(t(1+\varepsilon_1)\right)^{N(x,t\rho)(1+200\varepsilon_1) + C_3} \sup \limits_{B(x,\rho)}|u|.$$
 Noting that $t^{10\varepsilon_1} \geq (1+\varepsilon_1) $, since $t>2$, we can estimate
$$\left(t(1+\varepsilon_1)\right)^{N(x,t\rho)(1+200\varepsilon_1) + C_3} \leq t^{N(x,t\rho)(1+500\varepsilon_1) + C_8}.$$
 The proof of \eqref{n2} is finished.

 The inequality \eqref{n2*} follows immediately   from \eqref{n2} if we apply it to twice smaller $\varepsilon$, require
$N(x,t\rho) > \frac{2}{\varepsilon} C(\varepsilon/2,M,g)$ and put a  new smaller $R= R(\varepsilon/2,M,g)$.
\end{proof}
\begin{lemma} \label{l6}
 There exist  $r=r(M,g,O)$ and $N_0=N_0(M,g,O)$ such that for any points $x_1,x_2 \in B(O,r)$ and $\rho$ such that $N(x_1,\rho) > N_0 $ and $d(x_1,x_2)< \rho < r$, there exists $C=C(M,g,O)>0$ such that 
 $$N(x_2, C \rho) > \frac{99}{100} N(x_1, \rho).$$ 

\end{lemma}
\begin{proof}
One can choose such numerical $C$ that $B(x_2, C\rho) \supset B(x_1, C\rho(1-1/10^{10})) $ and $B(x_2, C\rho/2(1-10^{-9})) \subset B(x_1, C\rho(1-1/10^{10})) $.
 It follows from \eqref{n1*} and \eqref{n2*} that if we choose $t$ and $N_0$ properly, then  
$$ 2^{N(x_2,C\rho)(1+1/1000)}\geq \frac{\sup\limits_{B(x_2, C\rho)}|u|}{\sup\limits_{B(x_2, C\rho/2(1-10^{-9}) )}|u|} $$ 

$$\geq \frac{\sup\limits_{B(x_1, C\rho(1-10^{-10}))}|u|}{\sup\limits_{B(x_1, C\rho/2(1-10^{-10}) )}|u|} \geq 2^{N(x_1,\rho)(1-1/1000)}.$$
 Thus
 $$N(x_2, C\rho)> \frac{999}{1001}N(x_1,\rho)>\frac{99}{100}N(x_1,\rho).$$

\end{proof}

\end{document}